%% file: main.tex
\documentclass[11pt,reqno]{article}
\input{preamble.tex}

\numberwithin{equation}{section}
\usepackage{mathtools}
\usepackage{booktabs,multirow}
\usepackage{bbm}

\usepackage{multirow}
\usepackage{rotating,array}
\newcolumntype{L}{>{$}l<{$}}
\newcolumntype{C}{>{$}c<{$}}
\newcolumntype{R}{>{$}r<{$}}
\usepackage{multirow}
\begin{document}
\title{Shallow Characters and Supercuspidal Representations}
\author{Stella Sue Gastineau\\\texttt{stellasuegastineau@gmail.com}}
\date{\today}
\maketitle
\begin{abstract}
In 2014, Reeder and Yu constructed epipelagic representations of a reductive $p$-adic group $G$ from stable functions on shallowest Moy-Prasad quotients. In this paper, we extend these methods when $G$ is split. In particular, we classify all complex-valued characters vanshing on a slightly deeper Moy-Prasad subgroup and show that, while sufficient, a naive extension of Reeder-Yu's stability condition is not necessary for constructing supercuspidal representations.
\end{abstract}

\section{Introduction}
\subsection{Notation}
Let $k$ be a non-archimedean local field with value group $\bbZ$ and ring of integers $\frako$ with prime ideal $\frakp$ and residue field $\frakf:=\frako/\frakp$ of finite cardinality $q$ and characteristic $p$. Let $K$ be a maximal unramified extension of $k$, with algebraically closed residue field $\frakF$. Let $\bfG$ be an absolutely simple, simply connected algebraic group defined and splitting over $k$. We fix the following subgroups of $\bfG$ for consideration:
\begin{itemize}
\item $\bfT$ a maximal torus, defined and splitting over $k$.
\item $\bfB$ a Borel subgroup of $\bfG$, containing $\bfT$ and defined over $k$.
\item $\bfU$ the unipotent radical of $\bfB$, defined over $k$.
\end{itemize}
We will also use unbolded letters $G,B,T,U$ to denote the $k$-rational points of $\bfG,\bfB,\bfT,\bfU$ respectively. We will be assuming the basic structure of such groups, which can be found in \cite{IwahoriMatsumoto65,Tits79}.

\subsection{Motivation}
The group $G$ acts on its Bruhat-Tits building $\calB=\calB(\bfG,k)$ and for each point $\lambda\in\calB$, the stabilizer $P:=G_\lambda$ has a filtration by open {\bf Moy-Prasad subgroups}:
\[
P> P_{r_1}> P_{r_2}>\cdots
\]
indexed by an increasing, discrete sequece $\bfr(\lambda)=(r_1,r_2,\dots)$ of positive real numbers. The first Moy-Prasad subgroup $P_{r_1}$ is called the {\bf pro-unipotent radical} of $P$, and will be denoted by $P_+$. In their papers, Gross-Reeder \cite{GrossReeder10} and Reeder-Yu \cite{ReederYu14} study complex characters of 
\[
\chi:P_+\to\bbC^\times
\] 
that are trivial on the Moy-Prasad subgroup $P_{r_2}$.
In this paper we will go a little bit deeper down the Moy-Prasad filtration and classify all {\bf shallow characters}, those being characters that are trivial on Moy-Prasad subgroup $P_1\subset P_{r_2}$. 

\newpage
In \S\ref{subsec: shallow characters}, we show that a shallow character on $P_+$ can be recovered from its restrictions to its affine root subgroups and extended to a group homomorphism. In particular, in Theorem~\ref{main theorem} we show that in order to to define a shallow character, it is both necessary and sufficient that the extension be trivial on commutators
\[
[U_\beta,U_\alpha]\subset\prod_{i,j>0}U_{i\alpha+j\beta}
\]
where $\alpha$ and $\beta$ are affine roots whose gradients are not linearly dependant. 

Following a classification of shallow characters, we ask for which shallow characters $\chi:P_+\to\bbC^\times$ is the compactly-induced representation
\[
\ind_{P_+}^G(\chi)={\small\left\{\phi:G\to \bbC~\left|~ \begin{gathered}\text{$\phi(hx)=\chi(h)\cdot\phi(x)$}\\\text{$\phi$ compactly supported}\end{gathered}\right\}\right.}
\]
a supercuspidal representation of $G$. In their papers, Gross-Reeder and Reeder-Yu give a classification of supercuspidal representations of $G$ via stable orbits in a related graded Lie algebra. In Proposition~\ref{ReederYu type prop} of \S\ref{subsec: naive extension}, we look at a naive generalization of \cite[Propositio~2.4]{ReederYu14} and show that it is sufficient but not necessary for determining which shallow characters induce up to supercuspidal representations of $G$.


\section{Shallow Characters}
Throughout this paper we will fix an alcove of the apartment $\calA\subset\calB$ corresponding to $T$, and we will let 
\[
\Delta=\{\alpha_0,\alpha_1,\dots,\alpha_\ell\}
\]
denote the corresponding set of simple affine roots. We will also fix a point $\lambda$ contained in the closure of this alcove. We will denote by $\calF_J\subset\calA$ the facet containing $\lambda$ given by the non-vanishing of the simple affine roots $\Delta_J\subset\Delta$, where
\[
J\subsetneq\{0,1,,\dots,\ell\}.
\]
We will also let $P=G_\lambda$ denote the stabilizer of $\lambda$ in $G$.
\subsection{Shallow affine roots}
Given an affine root $\alpha:\calA\to\bbR$, we say that its {\bf depth} (at $\lambda$) is the real number $\alpha(\lambda)$. Then we say that $\alpha$ is {\bf shallow} (at $\lambda)$ if its depth is strictly between $0$ and $1$. We also say that $\alpha$ is {\bf decomposable} (as a shallow affine root) if there exists another shallow affine root $\beta$ such that $\alpha-\beta$ is a shallow affine root. Otherwise, we say that $\alpha$ is {\bf indecomposable} (as a shallow affine root).

Note that the depth of a shallow affine root precisely depends on $\lambda$; whereas, the set of decomposable and indecomposable shallow affine roots depend only on the the facet $\calF_J$ and not on the point $\lambda$ itself. In fact, setting
\[
n_J(n_0\alpha_0+n_1\alpha_1+\cdots+n_\ell\alpha_\ell):=\sum_{j\in J}n_j
\]
for $n_j\in\bbZ$, we can characterize the indecomposable shallow affine roots as follows:

\begin{lemma}
\label{decomposable shallow affine root}
A shallow affine root $\alpha$ is indecomposable if and only if $n_J(\alpha)=1$.
\end{lemma}
\begin{proof}
Let $\alpha$ be a shallow root. First note that if $n_J(\alpha)=1$, then $\alpha$ must be indecomposable as a shallow affine root: Indeed, if $\beta,\alpha-\beta$ is an affine root, then exactly one of $\alpha-\beta$ and $\beta$ is shallow since
\[
n_J(\alpha-\beta)=n_J(\alpha)-n_J(\beta).
\]
Therefore, for the remainder of the proof we suppose that $n_J(\alpha)\geq 2$.

First write
\[
\alpha=\alpha_{i_1}+\alpha_{i_2}+\cdots+\alpha_{i_m},
\]
so that 
\begin{align*}
\beta_j&=\alpha_{i_1}+\cdots+\alpha_{i_j}\\
\alpha-\beta_j&=\alpha_{i_{j+1}}+\cdots+\alpha_{i_m}
\end{align*}
are an affine roots for all $j=1,2,\dots,m$. Such a decomposition is possible, for example, by Lemma~3.6.2 in \cite{Carter89}. Since $n_J(\alpha)\geq 2$, we know that there exists a $j=1,2,\dots,m$ such that both $\beta_j$ and $\alpha-\beta_j$ are shallow. For instance, we can choose $j$ to be minimal such that $\alpha_{i_j}$ is a shallow affine root in $\Delta_J$. Thus, by setting $\beta=\beta_j$, we have given a decomposition
\[
\alpha=\beta+(\alpha-\beta)
\]
as shallow affine root whenever $n_J(\alpha)\geq 2$.
\end{proof}

\begin{lemma}
\label{shallow affine root lemma 1}
Suppose that $\alpha,\beta$ are shallow affine roots such that there are positive integers $i,j>0$ such that $i\alpha+j\beta$ are shallow affine roots. Then $\alpha+\beta$ is a shallow affine root.
\end{lemma}
\begin{proof}
Suppose that $i\alpha+j\beta$ is a shallow affine root for positive integers $i,j>0$. If both $i,j=1$, then $\alpha+\beta$ is a shallow affine root and we are done. Therefore, without loss of generality, we will assume that $i>0$. Note that in this case, we then have the following chain of inequalities:
\begin{equation}
\label{shallow root depth inequality}
0<\alpha(\lambda)<\alpha(\lambda)+\beta(\lambda)< i\alpha(\lambda)+j\beta(\lambda)<1.
\end{equation}
Thus, if we can show that $\alpha+\beta$ is an affine root, then it must be shallow. 

First, we note that $\alpha+\beta$ cannot be a constant function. Since $G$ is split, the minimal relation of the affine root group is of the form
\[
1=m_0\alpha_0+m_1\alpha_1+\cdots+m_\ell\alpha_\ell.
\]
In particular, if $\alpha+\beta$ was a positive constant function, then it must take value at least $1$. But this contradicts the inequalities in \eqref{shallow root depth inequality}.

Let $a,b$ be the respective gradients of $\alpha,\beta$. The subroot system of $R$ generated by $a$ and $b$ must have rank at most $2$. In fact, its rank must be exactly $2$, since $\alpha+\beta$ is not a constant function. We know that this rank~2 subsystem is not of type $\bfA_2$, since we are assuming that $ia+jb$ is a root for $i>1$. Therefore, we only need to consider the case where $a$ and $b$ generate a root system of type $\bfC_2$ or $\bfG_2$. In both cases, one can check directly that if $ia+jb$ is a root for positive integers $i,j$ then $a+b$ is a root as well.
\end{proof}

\subsection{Shallow characters}
\label{subsec: shallow characters}
A {\bf shallow character} of the pro-unipotent radical $P_+\subset P$ is any group homomorphism 
\[
\chi:P_+\to\bbC^\times
\]
that is trivial on the the Moy-Prasad subgroup 
\[
P_1=\left\langle T_0,U_\alpha\mid \alpha(\lambda)\geq 1\right\rangle,
\]
where $T_0=\bfT(1+\frakp)$ is the maximal compact subgroup of $T$ and $U_\alpha$ is the affine root subgroup of $G$ corresponding to the affine root $\alpha$. Since $P_1$ is a normal subgroup of $P_+$, any shallow character of $P_+$ must factor through the quotient $P_+/P_1$, a finite group generated by subgroups
\[
U_\alpha P_1/P_1\cong U_\alpha/U_{\alpha+1}\cong\frakf
\]
with $\alpha$ being shallow affine roots. Indeed, given any coset $gP_1$ in $P_+/P_1$, there is a unique decomposition
\begin{equation}
\label{unique coset decomposition}
gP_1=\left(\prod_{\alpha}u_{\alpha}(x_\alpha)P_1\right),
\end{equation}
where the product is relative to some fixed order over all shallow affine roots $\alpha$~\cite[\S3.1.1]{Tits79}. Therefore, any shallow character $\chi$ can be recovered from its restriction to the shallow affine root groups via the formula:
\[
\chi(gP_1):=\prod_{\alpha}\chi_{\alpha}(\bar x_\alpha),
\]
where $\chi_\alpha:\frakf\to\bbC^\times$ is the additive character defined by setting
\[
\chi_\alpha(\bar x):=\chi(u_\alpha(x)P_1)
\]
for any lift $x\in\frako$ of $\bar x\in\frakf$.

\begin{lemma}
\label{shallow character commutator lemma}
Let $\chi:P_+/P_1\to\bbC^\times$ be a shallow character of $P_+$ given by additive characters as above. Then for any shallow affine roots $\alpha,\beta$ we have the following identities:
\[
1=\prod_{i,j}\chi_{i\alpha+j\beta}(C_{\alpha\beta ij}\bar x^i\bar y^j),
\]
where the product is over all $i,j>0$ such that $i\alpha+j\beta$ is a shallow affine root and the constants $C_{\alpha\beta ij}$ are given as in the Chevalley Commutator Formula \cite[Theorem~5.2.2]{Carter89}. 
\end{lemma}
\begin{proof}
Let $\alpha,\beta$ be two shallow affine roots such that $i\alpha+j\beta$ is a shallow affine root for some positive integers $i,j>0$. Then by Lemma~\ref{shallow affine root lemma 1}, we know that $\alpha+\beta$ is a shallow affine root. Therefore, we can apply the Chevalley commutator formula \cite[Theorem~5.2.2]{Carter89}, which says that
\[
[u_\beta(y),u_\alpha(x)]P_1=\prod_{i,j}u_{i\alpha+j\beta}(C_{\alpha\beta ij}x^iy^j)P_1
\]
for all $x,y\in\frako$. Here the product is in increasing order over all $i,j>0$ such that $i\alpha+j\beta$ is an affine root. But if any $i\alpha+j\beta$ is not shallow, then $U_{i\alpha+j\beta}\subset P_1$. Therefore, we can assume that the product is only over $i,j>0$ such that $i\alpha+j\beta$ is a shallow affine root. 

Now let $\chi:P_+/P_1\to\bbC^\times$ be any shallow character of $P_+$. Since $\chi$ is a group homomorphism, we know that
\begin{align*}
\chi([u_\beta(y),u_\alpha(x)]P_1)
&=\prod_{i,j}\chi(u_{i\alpha+j\beta}(C_{\alpha\beta ij}x^iy^j)P_1)\\
&=\prod_{i,j}\chi_{i\alpha+\beta}(C_{\alpha\beta ij}\bar x^i\bar y^j)
\end{align*}
where the product is over all $i,j>0$ such that $i\alpha+j\beta$ is a shallow affine root. Finally, since $\chi$ maps into an abelian group $\bbC^\times$, we know that
\[
\chi([u_\beta(y),u_\alpha(x)]P_1)=1,
\]
finishing our proof.
\end{proof}

\begin{theorem}
\label{main theorem}
Suppose that for each shallow affine root $\alpha$, we are given an additive character $\chi_\alpha:\frakf\to\bbC^\times$. Suppose further that for each pair of shallow affine roots $\alpha,\beta$ we have the following relation:
\begin{equation}
\label{shallow character commutator relation}
1=\prod_{i,j}\chi_{i\alpha+j\beta}(C_{\alpha\beta ij}\bar x^i\bar y^j),
\end{equation}
where the product is over all $i,j>0$ such that $i\alpha+j\beta$ is a shallow affine root. Then there exists a unique shallow character $\chi:P_+/P_1\to\bbC^\times$ such that
\begin{equation}
\label{shallow affine root character}
\chi(u_{\alpha}(x)P_1)=\chi_{\alpha}(\bar x)
\end{equation}
for all $x\in\frako$ and shallow affine root $\alpha$. Moreover, any shallow affine root is of this form.
\end{theorem}
\begin{proof}
For the proof of this theorem, we will fix an enumeration of the shallow affine roots $\alpha_1,\dots,\alpha_n$ so that $i<j$ whenever $\alpha_i(\lambda)<\alpha_j(\lambda)$. Then we construct the well-defined function $\chi:P_+/P_1\to\bbC^\times$ by setting
\begin{equation}
\label{definition of chi}
\chi\left(\prod_{i=1}^nu_{\alpha_i}(x_i)P_1\right):=\prod_{i=1}^n\chi_{\alpha_i}(\bar x_i)
\end{equation}
for all $x_1,\dots,x_n\in\frako$. Indeed, this function is well-defined since each coset in $P_+/P_1$ has a unique decomposition of the form \eqref{unique coset decomposition} with respect to this shallow affine root ordering. What follows is a proof that $\chi$ defines a group homomorphism, and thus, is the unique shallow character satisfying \eqref{shallow affine root character}. Since $P_+/P_1$ is generated by the subgroups $U_\alpha P_1/P_1$ for shallow affine roots, it will be sufficient to show that 
\begin{equation}
\label{chi is a character}
\chi(gu_\alpha(x)P_1)=\chi(gP_1)\cdot\chi_\alpha(\bar x)
\end{equation}
for all cosets $gP_1$ in $P_+/P_1$ and all shallow affine roots $\alpha$.

Let $\alpha=\alpha_j$ be a shallow affine root. We now show that \eqref{chi is a character} holds via descending induction on $j$. For the base case, we let $j=n$ so that
\begin{align*}
\chi\left(\left[\prod_{i=1}^nu_{\alpha_i}(x_i)P_1\right]u_{\alpha_n}(x)P_1\right)
&=\chi\left(\left[\prod_{i=1}^{n-1}u_{\alpha_i}(x_i)P_1\right]u_{\alpha_n}(x_n+x)P_1\right)\\
&=\left[\prod_{i=1}^{n-1}\chi_{\alpha_i}(\bar x_i)\right]\cdot \chi_{\alpha_n}(\bar x_n+\bar x)\\
&=\left[\prod_{i=1}^{n}\chi_{\alpha_i}(\bar x_i)\right]\cdot\chi_{\alpha_n}(\bar x)
\end{align*}
for all $x_1,\dots,x_n,x\in\frako$. For the induction step, assume that
\[
\chi(gu_{\alpha_i}(x)P_1)=\chi(gP_1)\cdot \chi_{\alpha_i}(\bar x)
\]
for all cosets $gP_1$ in $P_+/P_1$ and every shallow affine root $\alpha_i$ with $i>j$. In this case, we look at products of the form
{\small\[
\left[\prod_{i=1}^nu_{\alpha_i}(x_i)P_1\right]u_{\alpha_j}(x)P_1=\left[\prod_{i=1}^{j-1}u_{\alpha_i}(x_i)P_1\right] u_{\alpha_j}(x_j+x)P_1\left[\prod_{i=j+1}^nu_{\alpha_i}(x_i)[u_{\alpha_j}(x),u_{\alpha_i}(x_i)]P_1\right]
\]}
If $\alpha_i+\alpha_j$ is a constant, then 
\[
[u_{\alpha_j}(x),u_{\alpha_i}(x_i)]P_1=P_1.
\]
Otherwise, we can use the Chevalley commutator formula to say that
\[
[u_{\alpha_j}(x),u_{\alpha_i}(x_i)]P_1=\prod_{k,l}u_{k\alpha_j+l\alpha_i}(C_{\alpha_j\alpha_i kl}x^kx_i^l)P_1
\]
where the product is in increasing order over over all $k,l>0$ such that $k\alpha_j+l\alpha_i$ is a shallow affine root. Note that each such $k\alpha_j+l\alpha_i$ must occur later than $\alpha_i$ in the enumeration of shallow affine roots since $k\alpha_j(\lambda)+l\alpha_i(\lambda)>\alpha_j(\lambda)$. By repeadily applying the induction hypothesis and using relation \eqref{shallow character commutator relation}, we have that 
\begin{equation}
\label{shallow character commutator relation intermediate step}
\chi(g[u_{\alpha_j}(x),u_{\alpha_i}(x_i)]P_1)=\chi(gP_1)\left(\prod_{k,l}\chi_{k\alpha_j+l\alpha_i}(C_{\alpha_j\alpha_ikl}\bar x^k\bar x_i^l)\right)=\chi(gP_1)
\end{equation}
for all cosets $gP_1$ in $P_+/P_1$. 
Thus, repeatedly applying the induction hypothesis and \eqref{shallow character commutator relation intermediate step}, we have
\begin{multline*}
\chi\left(\left[\prod_{i=1}^nu_{\alpha_i}(x_i)P_1\right]u_{\alpha_j}(x)P_1\right)\\
\begin{aligned}
&=\chi\left(\left[\prod_{i=1}^{j-1}u_{\alpha_i}(x_i)P_1\right] u_{\alpha_j}(x_j+x)P_1\left[\prod_{i=j+1}^nu_{\alpha_i}(x_i)[u_{\alpha_j}(x),u_{\alpha_i}(x_i)]P_1\right]\right)\\
&=\chi\left(\left[\prod_{i=1}^{j-1}u_{\alpha_i}(x_i)P_1\right] u_{\alpha_j}(x_j+x)P_1\left[\prod_{i=j+1}^{n-1}u_{\alpha_i}(x_i)[u_{\alpha_j}(x),u_{\alpha_i}(x_i)]P_1\right]u_{\alpha_n}(x_n)P_1\right)\\
&=\chi\left(\left[\prod_{i=1}^{j-1}u_{\alpha_i}(x_i)P_1\right] u_{\alpha_j}(x_j+x)P_1\left[\prod_{i=j+1}^{n-1}u_{\alpha_i}(x_i)[u_{\alpha_j}(x),u_{\alpha_i}(x_i)]P_1\right]\right)\cdot\chi_{\alpha_n}(\bar x_n)\\
&\vdotswithin{=}\\
&=\chi\left(\left[\prod_{i=1}^{j-1}u_{\alpha_i}(x_i)P_1\right] u_{\alpha_j}(x_j+x)P_1\right)
\left[\prod_{i=j+1}^{n-1}\chi_{\alpha_i}(\bar x_i)\right].\\
\end{aligned}
\end{multline*}
Finally, using the definition of $\chi$ given in \eqref{definition of chi}, we arrive at
\begin{align*}
\chi\left(\left[\prod_{i=1}^nu_{\alpha_i}(x_i)P_1\right]u_{\alpha_j}(x)P_1\right)
&=\chi\left(\left[\prod_{i=1}^{j-1}u_{\alpha_i}(x_i)P_1\right] u_{\alpha_j}(x_j+x)P_1\right)
\left[\prod_{i=j+1}^{n-1}\chi_{\alpha_i}(\bar x_i)\right]\\
&=\left[\prod_{i=1}^{j-1}\chi_{\alpha_i}(\bar x_i)\right] \chi_{\alpha_j}(\bar x_j+\bar x)
\left[\prod_{i=j+1}^{n-1}\chi_{\alpha_i}(\bar x_i)\right]\\
&=\left[\prod_{i=1}^n\chi_{\alpha_i}(\bar x_i)\right]\chi_{\alpha_j}(\bar x)
\end{align*}
for all $x_1,\dots,x_n,x\in\frako$ as desired.

This finishes our proof that there is a unique shallow character of $P_+$ satisfying \eqref{shallow affine root character}. To see that every shallow character of $P_+$ is of this form, we note Lemma~\ref{shallow character commutator lemma} says that its restrictions to shallow affine root groups must satisfy \eqref{shallow character commutator relation}.
\end{proof}

\begin{corollary}
\label{corollary indecomposable shallow character}
Suppose that for each shallow affine root $\alpha$, we are given an additive character 
\[
\chi_\alpha:\frakf\to\bbC^\times.
\]
Suppose further that $\chi_\alpha$ is trivial whenever $\alpha$ is decomposable as a shallow affine root. Then there exists a unique shallow character $\chi:P_+/P_1\to\bbC^\times$ such that
\[
\chi(u_\alpha(x)P_1)=\chi_\alpha(\bar x)
\]
for all $x\in\frako$ and shallow affine roots $\alpha$. 
\end{corollary}
\begin{proof}
By the previous theorem, we only need to show that given any shallow affine roots $\alpha,\beta$ we have the following relations:
\begin{equation}
\label{corollary indecomposable shallow character equation}
1=\prod_{i,j}\chi_{i\alpha+j\beta}(C_{\alpha\beta ij}\bar x^i\bar y^j),
\end{equation}
where the product is in increasing order over all $i,j>0$ such that $i\alpha+j\beta$ is a shallow affine root. But this is true because each $i\alpha+j\beta$ is a decomposable shallow affine root, and thus each $\chi_{i\alpha+j\beta}$ is trivial. Thus \eqref{corollary indecomposable shallow character equation} naturally holds.
\end{proof}

\subsection{The space of shallow characters}
\label{sec: space of shallow characters}
Let $\check\sfV$ be the set of all shallow characters of $P_+$. Then $\check\sfV$ has a natural abelian group structure given by 
\[
(\chi_1+\chi_2)(g)=\chi_1(g)\cdot\chi_2(g).
\]
Moreover, the group $\check\sfV$ can be endowed with the structure of a $\frakf$-vector space as shown below: The finite group $P_+/P_1$ is generated by subgroups of the form
\[
U_\alpha P_1/P_1\cong U_\alpha/U_{\alpha+1}\cong\frakf
\]
for shallow affine roots $\alpha$. Once a pinning of $G$ has been chosen, there is a natural action of $\frakf$ on each of these subgroups by setting
\[
\bar z\cdot u_\alpha(x)P_1:=u_\alpha(zx)P_1
\]
for all $x,z\in\frako$ and shallow affine roots $\alpha$.  This action can be extended to the full group $P_+/P_1$ via distribution by setting
\[
\bar z\cdot (u_\alpha(x)u_\beta(y)P_1)=u_\alpha(zx)u_\beta(zy)P_1
\]
for all $x,y,z\in\frako$ and shallow affine roots $\alpha,\beta$. This in turn endows the abelianization
\[
\sfV:=\frac{P_+/P_1}{[P_+/P_1,P_+/P_1]}
\]
with the structure of a $\frakf$-vector space spanned by vectors $\sfv_\alpha$, the image of $u_\alpha(1)P_1$ under the quotient $P_+/P_1\to\sfV$. Finally, this action endows $\check\sfV$ with the structure of a $\frakf$-vector space with $\frakf$-action given via
\[
[\bar z\cdot\chi](gP_1):=\chi(\bar z^{-1}\cdot gP_1).
\]
Thus, we have shown that $\check\sfV$ is a $\frakf$-vector space.

\subsubsection{Epipelagic characters}
Recall that for real number $0<r<1$, we say that a shallow affine root $\alpha$ has depth $r$ provided that $\alpha(\lambda)=r$. We now say that a shallow character $\chi\in\check\sfV$ has {\bf depth} $r$ provided that the following hold:
\begin{itemize}
\item $\chi_\alpha$ is non-trivial for some shallow affine root $\alpha$ of depth $r$.
\item $\chi_\alpha$ is trivial for all shallow affine roots $\alpha$ of depth greater than $r$.
\end{itemize}
The minimal depth $\alpha(\lambda)=r$ for shallow affine roots $\alpha$ is $r=r_1$, the index of the pro-unipotent radical $P_+=P_{r_1}$ in the Moy-Prasad filtration. The affine roots at this depth are said to be {\bf epipelagic}, and since any epipelagic affine root is necessarily indecomposable as a shallow affine root, Corollary~\ref{corollary indecomposable shallow character} implies that the set of all shallow characters of depth $r_1$ form a non-trivial subspace of $\check\sfV$, denoted 
\[
\check\sfV_+:=\check\sfV_{r_1},
\]
whose dimension is equal to the non-zero number of epipelagic affine roots. More generally, for all real numbers $0<r<1$, we let
\[
\check\sfV_r:=\{\chi\in\sfV\mid\text{$\chi$ is trivial on $P_s$ for all $s>r$}\}
\]
be the subspace of all shallow characters of depth at most $r$.

\section{Supercuspidal Representations}
Recall that a {\bf smooth representation} of $G$ is a group homomorphism
\[
\pi:G\to\GL(V),
\]
where $V$ is a complex vector space, such that for every $v\in V$ there is a compact open subgroup $H\subset G$ such that $\pi(g)v=v$ for every $g\in H$. We say that a smooth representation $\pi$ is {\bf supercuspidal} is every matrix coefficient of $G$ is compactly supported modulo the center $Z(G)$. We will now investigate which shallow characters of $P_+$ give rise to supercuspidal representations of $G$ via compact induction.

\subsection{Compact Induction}
In this section we will recall some basic facts about compact induction: Let $\chi:P_+/P_1\to\bbC^\times$ be a shallow character of $P_+$, and consider the {\bf compactly-induced representation} of $G$
\[
\pi(\chi):=\ind_{P_+}^G(\chi)={\small\left\{\phi:G\to \bbC~\left|~ \begin{gathered}\text{$\phi(hg)=\chi(h)\cdot\phi(g)$}\\\text{$\phi$ compactly supported}\end{gathered}\right\}\right.},
\]
with $G$-action given by right translations:
\[
[n\cdot\phi](g):=\phi(gn)
\]
for all $n,g\in G$. Given any $n\in G$, we set ${}^nP_+:=nP_+n^{-1}$ and let ${}^n\chi$ be the conjugate character of ${}^nP_+$ given by setting
\[
{}^n\chi(g):=\chi(n^{-1}gn)
\]
for all $g\in{}^nP_+$. We then define the {\bf intertwining set} to be 
\[
\scrI(G,P_+,\chi):=\{n\in G\mid\text{${}^n\chi\cong \chi$ on ${}^nP_+\cap P_+$}\}.
\]
Then we have the following basic result:
\begin{lemma}
Let $\chi:P_+/P_1\to\bbC^\times$ be a shallow character of $P_+$. Then the following are equivalent:
\begin{enumerate}[label=\alph*.]
\item $\scrI(G,P_+,\chi)=P_\chi$.
\item $\pi(\chi)$ is irreducible.
\item $\pi(\chi)$ is supercuspidal.
\end{enumerate}
\end{lemma}

Recall that the parahoric subgroup $P$ normalizes Moy-Prasad subgroups $P_+,P_1$, and so the conjugate character ${}^n\chi$ is then a shallow charater of $P_+$ for any $n\in P$. We therefore consider the stabilizer of $\chi$ in $P$:
\[
P_\chi:=\{n\in N\mid{}^n\chi=\chi\}\subset \scrI(G,P_+,\chi). 
\]
The finite quotient $P_\chi/P_+$ has order equal to the dimension of the semisimple {\bf intertwining algebra} 
\[
\scrA_\chi:=\End_{P_\chi}(\ind_{P_+}^{P_\chi}(\chi)).
\]
There is a bijection $\rho\mapsto\chi_\rho$ between equivalence classes of irreducible $\scrA_\chi$-modules and the irreducible $P_\chi$ representations appearing in the isotypic decomposition
\[
\ind_{P_+}^{P_\chi}(\chi)=\bigoplus_\rho\dim(\rho)\cdot\chi_\rho.
\]
Then we have the following result, whose proof can be found in \cite[\S2.1]{ReederYu14}:
\begin{lemma}
\label{Mackey theory 2}
Let $\chi:P_+/P_1\to\bbC^\times$ be a shallow character of $P_+$. If $\scrI(G,P_+,\chi)=P_\chi$, then we have the following isotypic decomposition:
\[
\pi(\chi)=\bigoplus_{\rho}\dim (\rho)\cdot \ind_{P_\chi}^G(\chi_\rho),
\]
where the direct sum is over all simple $\scrA_\chi$ modules $\rho$. Moreover, each compactly induced representation
\[
\pi(\chi,\rho):=\ind_{P_\chi}^G(\chi_\rho)
\]
are inequivalent irreducible supercuspidal representations of $G$.
\end{lemma}

\subsection{Supercuspidal representations coming from shallow characters}
\label{subsec: naive extension}
Let $\mu$ be any point in the apartment $\calA$. For all positive real numbers $s>0$, let
\[
\sfV_{\mu,s}:=\Span_\frakf\{\sfv_\alpha\in\sfV\mid\text{$0<\alpha(\lambda)<1$ and $\alpha(\mu)\geq s$}\}
\]
be the $\frakf$-span of the vectors $\sfv_\alpha$ for shallow affine roots $\alpha$ such that $\alpha(\mu)\geq s$. Then we have the following sufficient condition for constructing supercuspidal representations:
\begin{proposition}
\label{ReederYu type prop}
Let $\chi\in\check\sfV_{r}$ be any depth $r$ shallow character such that the following holds:
\begin{itemize}
\item[$(*)$] If $n\in N_G(T)$ and $\chi$ identically vanishes on $\sfV_{n\lambda,s}$ for all $s>r$, then $n\lambda=\lambda$.
\end{itemize}
Then $\scrI(G,P_+,\chi)=P_\chi$.
\end{proposition}
\begin{proof}
Let $\chi\in\check\sfV_r$ be a depth $r$ shallow character of $P_+$ satisfying $(*)$. Since $P$ contains an Iwahori subgroup, the affine Bruhat decomposition \cite{IwahoriMatsumoto65} implies that in order to show that $\scrI(G,P_+,\chi)=P_\chi$, it will be sufficient to consider $n\in N_G(T)$ and show that if 
\begin{equation}\label{Mackey eqn}
{}^n\chi=\chi\text{ on }{}^nP_+\cap P_+.
\end{equation}
then $n\in P$.

Let $n\in N_G(T)$ be such that \eqref{Mackey eqn} holds, and fix a real number $s>r$. It is certainly true that 
\begin{equation}
\label{Mackey eqn 2}
{}^n\chi=\chi\text{ on }{}^nP_{s}\cap P_+
\end{equation}
for the Moy-Prasad subgroup $P_s\subset P$. Let $\alpha$ be any shallow root such that $\alpha(n\lambda)\geq s$. Since it has depth $r$, $\chi$ must then be trivial on $U_{n^{-1}\alpha}\subset P_s$. Therefore, $\chi_\alpha$ must be the trivial additive character, since \eqref{Mackey eqn 2} requires that
\[
\chi_{\alpha}(\bar x)=\chi(u_{\alpha}(x))={}^n\chi(u_{\alpha}(x))=\chi(u_{n^{-1}\alpha}(\pm x))=1
\]
for all $x\in\frako$. But this holds for all $s>r$ and all shallow affine roots $\alpha$ such that $\alpha(n\lambda)\geq s$, and thus $\chi$ vanishes identically on $\sfV_{n\lambda,s}$ for all $s>r$. Consequently, $(*)$ implies that $n\lambda=\lambda$ so that $n\in P$.
\end{proof}

\begin{remark}
In the remainder of this subsection we study condition $(*)$ of Proposition~\ref{ReederYu type prop} in further detail. In particular, we first show in \S\ref{subsub simple supercuspidal} how $(*)$ is a necessary condition for constructing simple supercuspidal representations of $G$. Then in \S\ref{sec: Sp4} we show how, when leaving the epipelagic case, condition $(*)$ is no longer necessary for constructing supercuspidal representations of $G$.
\end{remark}

\subsubsection{Simple supercuspidal representations}
\label{subsub simple supercuspidal}
In this subsubsection only, we will make the additional assumption that $\lambda$ is the barycenter of the fundamental open alcove in $\calA$ bonded by $\Delta$. If 
\begin{equation}
\label{minimal relation simple supercuspidal}
1=m_0\alpha_0+m_1\alpha_1+\cdots+m_\ell\alpha_\ell
\end{equation}
is the minimal integral relation on simple affine roots with $m_i>0$, then $\lambda$ is the unique point such that for all simple $\alpha_i\in\Delta$, 
\[
\alpha_i(\lambda)=1/h,
\]
where $h:=m_0+m_1+\cdots+m_\ell$ is the Coxeter number of $R$. In this case, the parahoric subgroup $P=G_\lambda$ is an Iwahori subgroup of $G$.

\begin{lemma}
\label{stable epipelagic character lemma}
Let $\lambda$ be the barycenter of the fundamental open alcove in $\calA$. Then for any $n\in N_G(T)$ such that $n\lambda\neq \lambda$, there must exist a simple affine root $\alpha_i\in\Delta$ such that $\alpha_i(n\lambda)>1/h$. 
\end{lemma}
\begin{proof}
Let $n\in N_G(T)$ be such that $n\lambda\neq\lambda$. The difference $\mu=\lambda-n\lambda$ belongs to the translation group
\[
E:=\bbR\otimes_\bbZ\Hom(k,T),
\]
so that we can write $\mu=sc$ for some real number $s>0$ and non-trivial cocharacter $c\in\Hom(k,T)$. For all simple affine roots $\alpha_i\in\Delta$, we have 
\[
\alpha_i(n\lambda)=\alpha_i(\lambda+sc)=\alpha_i(\lambda)+s\langle a_i,c\rangle,
\]
where $a_i$ is the gradient of $\alpha_i$. Since $\Delta$ forms a base of the affine root system, the gradients $a_0,a_1,\dots,a_\ell$ form a spanning set of the $\ell$-dimensional vector space
\[
E^*:=\bbR\otimes_\bbZ\Hom(T,k),
\]
which is dual to $E$ under the natural pairing $\langle\cdot,\cdot\rangle$. Therefore, there must be some $\alpha_i$ such that $\langle a_i,c\rangle\neq 0$. Without loss of generality, we can assume that $\langle a_i,c\rangle >0$ so that $\alpha_i(n\lambda)>1/h$; otherwise, if $\langle a_j,c\rangle\leq 0$ for all $\alpha_j\in\Delta$, then \eqref{minimal relation simple supercuspidal} implies that
\[
0=m_0\langle a_0,c\rangle+m_1\langle a_1,c\rangle+\cdots+m_\ell\langle a_\ell,c\rangle<0,
\]
a contradiction.
\end{proof}

\begin{lemma}
\label{lemma long element}
Let $\lambda$ be the barycenter of the fundamental open alcove in $\calA$. Then given any non-empty, proper subset
\[
I\subsetneq\{0,1,\dots,\ell\},
\]
there must exist an element $n\in N_G(T)$ such that $\alpha_i(n\lambda)<0$ for all $i\in I$.
\end{lemma}
\begin{proof}
Consider the affine Weyl group
\[
W:=N_G(T)/T_0
\]
and the subgroup $W_I$ of $W$ generated by simple reflections along the simple affine roots $\alpha_i$ for $i\in I$. Note that $W_I$ is a non-empty, finite Coxeter group, since $I$ is a non-empty, proper subset of $\{0,1,\dots,\ell\}$. Let $w:=w_I$ be the long element in $W_I$; that is, $w$ is the unique element on $W_I$ such that $w\alpha_i$ is a negative affine root for all $i\in I$. Such an element has order $2$, so that
\[
w^{-1}\alpha=w\alpha
\]
for all affine roots $\alpha$. Moreover, since an affine root is negative if and only if it takes negative values on the open fundamental alcove, we have 
\[
\alpha_i(w\lambda)=(w^{-1}\alpha_i)(\lambda)=(w\alpha_i)(\lambda)<0
\]
for all $i\in I$. Thus, letting $n\in N_G(T)$ be any lift of $w$, we are done.
\end{proof}

\begin{proposition}
Let $\lambda$ be the barycenter of the fundamental open alcove in $\calA$. Then given any epipelagic character $\chi\in\check\sfV_{1/h}$, the following are equivalent:
\begin{enumerate}[label=\alph*.]
\item $\chi_{\alpha_i}$ is non-trivial for all $\alpha_i\in\Delta$.
\item If $n\in N_G(T)$ and $\chi$ vanishes identically on $\sfV_{n\lambda,s}$ for all $s>1/h$, then $n\lambda=\lambda$.
\end{enumerate}
\end{proposition}

\begin{proof}
$(a\Rightarrow b)$: Suppose that $\chi_{\alpha_i}$ is non-trivial for all $\alpha_i\in\Delta$, and let $n\in N_G(T)$. By Lemma~\ref{stable epipelagic character lemma} there exists some $\alpha_i$ such that $\alpha_i(n\lambda)>1/h$. Since $\chi_{\alpha_i}$ is non-trivial, there must exist some $s>1/h$ such that $\chi$ does not vanish identically on $\frakf\sfv_{\alpha_i}\subset\sfV_{n\lambda,s}$.

$(\neg a\Rightarrow \neg b)$: Suppose that there exists some simple affine root $\alpha_i\in\Delta$ such that $\chi_{\alpha_i}$ is trivial. 
Setting 
\[
I:=\{i\mid\text{$\chi_{\alpha_i}$ is non-trivial}\}\subsetneq\{0,1,\dots,\ell\}
\]
and applying Lemma~\ref{lemma long element}, we see that there must exist some $n\in N_G(T)$ such that $\alpha_i(n\lambda)<0$ whenever $\chi_{\alpha_i}$ is non-trivial. In this case, for all $s>1/h$, the vector space $\sfV_{n\lambda,s}$ is contained within the span of subspaces $\frakf\sfv_\alpha$ for shallow affine roots $\alpha$ such that $\chi_\alpha$ is trivial. Thus, $\chi$ identically vanishes on $\sfV_{n\lambda,s}$ while $n\lambda\neq\lambda$.
\end{proof}

\begin{corollary}
Let $\lambda$ be the barycenter of the fundamental open alcove in $\calA$, and let $\chi\in\check\sfV_{1/h}$ be any epipelagic character such that $\chi_{\alpha_i}$ is non-trivial for all $\alpha_i\in\Delta$. Then $\scrI(G,P_+,\chi)=P_\chi$.
\end{corollary}

\begin{remark}
In the case given by the above corollary, the supercuspidal representations $\pi(\chi,\rho)$ obtained from compact induction are called {\bf simple supercuspidal representations}, and they were first studied by Gross-Reeder in \cite{GrossReeder10}. This is a special class of epipelagic representations which were later studied by Reeder-Yu in \cite{ReederYu14}.
\end{remark}

\subsubsection{A supercuspidal representation of $\bfSp_4(\bbQ_2)$}
\label{sec: Sp4}
Let $G=\bfSp_4(k)$ be the simply connected Chevalley group consisting of matrices in $\bfSL_2(k)$ which are fixed under the endomorphism 
\[
X\mapsto Q^{-1}(X^{\dag})^{-1}Q,
\]
where $[x_{ij}]^\dag=[x_{ji}]$ denotes transposition and $Q$ is the skew-symmetric matrix
\[
Q=\begin{bmatrix}&&&1\\&&1\\&-1\\-1\end{bmatrix}.
\]
Alternatively, $G$ is seen as the group of isometries with respect to the Hermitian form given by $Q$. We fix the diagonal maximal torus 
\[
T=\left.\left\{t=\begin{bmatrix}t_1\\&t_2\\&&t_3\\&&&t_4\end{bmatrix}~\right|~\begin{gathered}\text{$t_1,t_2,t_3,t_4\in \bbQ_2^\times$ with}\\\text{$t_1t_4=1$ and $t_2t_3=1$}\end{gathered}\right\}
\]
The root system $R=R(G,T)$ of $G$ relative to $T$ has type $\bfC_2$
with base given by short root $a_1(t)=t_1/t_2$ and long root $a_2(t)=t_2/t_3$. For convenience, we will denote by $a_0(t)=t_4/t_1$ the lowest long root in $R$ relative to this chosen base. A base $\Delta$ of the affine root system of $G$ relative to $T$ can be given by the following three affine functionals:
\begin{align*}
\alpha_0&=a_0+1\\
\alpha_1&=a_1+0\\
\alpha_2&=a_2+0
\end{align*}
It should be noted that these simple affine roots satisfy the minimal relation
\[
\alpha_0+2\alpha_1+\alpha_2=1.
\]

By fixing a pinning of $G$ via the following root group morphisms:
\begin{align*}\small
u_{a_1}(x)&={\begin{bmatrix}1&x&&\\&1&&\\&&1&-x\\&&&1\end{bmatrix}}
&u_{a_0+a_1+a_2}(x)&={\begin{bmatrix}1&&&\\x&1&&\\&&1&\\&&-x&1\end{bmatrix}}\\
u_{a_2}(x)&={\begin{bmatrix}1&&&\\&1&x&\\&&1&\\&&&1\end{bmatrix}}
&u_{2a_1+a_0}(x)&={\begin{bmatrix}1&&&\\&1&&\\&x&1&\\&&&1\end{bmatrix}}\\
u_{a_1+a_2}(x)&={\begin{bmatrix}1&&x&\\&1&&x\\&&1&\\&&&1\end{bmatrix}}
&u_{a_0+a_1}(x)&={\begin{bmatrix}1&&&\\&1&&\\x&&1&\\&x&&1\end{bmatrix}}\\
u_{2a_2+a_1}(x)&={\begin{bmatrix}1&&&x\\&1&&\\&&1&\\&&&1\end{bmatrix}}
&u_{a_0}(x)&={\begin{bmatrix}1&&&\\&1&&\\&&1&\\x&&&1\end{bmatrix}}
\end{align*}
for $x\in k$, we are able to directly compute the structure constants in the Chevalley commutator formulas:
\begin{equation}\label{C2 commutators}
\left.
\begin{aligned}\small
{}[u_{\alpha_1}(y),u_{\alpha_2}(x)]&=u_{\alpha_1+\alpha_2}(+xy)u_{2\alpha_1+\alpha_2}(-xy^2)\\
[u_{\alpha_1}(y),u_{\alpha_0}(x)]&=u_{\alpha_0+\alpha_1}(-xy)u_{\alpha_0+2\alpha_1}(-xy^2)\\[8pt]
[u_{\alpha_1}(y),u_{\alpha_1+\alpha_2}(x)]&=u_{2\alpha_1+\alpha_2}(+2xy)\\
[u_{\alpha_1}(y),u_{\alpha_0+\alpha_1}(x)]&=u_{\alpha_0+2\alpha_1}(-2xy)\\[8pt]
[u_{\alpha_2}(y),u_{\alpha_0+\alpha_1}(x)]&=u_{\alpha_0+\alpha_1+\alpha_2}(-xy) u_{\alpha_0+1}(-x^2y)\\
[u_{\alpha_0}(y),u_{\alpha_1+\alpha_2}(x)]&=u_{\alpha_0+\alpha_1+\alpha_2}(-xy) u_{\alpha_2+1}(-x^2y)\\[8pt]
[u_{\alpha_1+\alpha_2}(y),u_{\alpha_0+2\alpha_1}(x)]&=u_{\alpha_1+1}(+xy) u_{2\alpha_1+\alpha_2+1}(+xy^2)\\
[u_{\alpha_0+\alpha_1}(y),u_{2\alpha_1+\alpha_2}(x)]&=u_{\alpha_1+1}(-xy) u_{\alpha_0+2\alpha_1+1}(+xy^2) \\[8pt]
[u_{\alpha_1+\alpha_2}(y),u_{\alpha_0+\alpha_1+\alpha_2}(x)]&=u_{\alpha_2+1}(-2xy)\\
[u_{\alpha_0+\alpha_1}(y),u_{\alpha_0+\alpha_1+\alpha_2}(x)]&=u_{\alpha_0+1}(+2xy)\\[8pt]
[u_{2\alpha_1+\alpha_2}(y),u_{\alpha_0+\alpha_1+\alpha_2}(x)]&=u_{\alpha_1+\alpha_2+1}(-xy)u_{\alpha_2+2}(+x^2y)\\
[u_{\alpha_0+2\alpha_1}(y),u_{\alpha_0+\alpha_1+\alpha_2}(x)]&=u_{\alpha_0+\alpha_1+1}(+xy)u_{\alpha_0+2}(+x^2y)
\end{aligned}\right\}
\end{equation}
for any $x,y\in\frako$.

Suppose that $\lambda$ is contained within the closure of the alcove bounded by the vanishing hyperplanes of the simple affine roots in $\Delta$. The set of positive affine roots which take value at most $1$ at $\lambda$ is therefore 
\[
\{\alpha_0,\alpha_1,\alpha_2,\alpha_0+\alpha_1,\alpha_1+\alpha_2,\alpha_0+2\alpha_1,\alpha_0+\alpha_1+\alpha_2,2\alpha_1+\alpha_2\},
\]
and those which take non-zero value at $\lambda$ form the shallow affine roots. Thus, in order to define a shallow character 
\[
\chi:P_+/P_1\to\bbC^\times,
\]
one only needs to verify that the restrictions to the shallow affine root groups satisfy the following relations coming from the commutators in \eqref{C2 commutators}:
\begin{equation}
\label{C2 character relations}
\begin{cases}
1=\chi_{\alpha_1+\alpha_2}(xy)\cdot \chi_{2\alpha_1+\alpha_2}(xy^2)&\text{if $\alpha_1,\alpha_2$ are shallow}\\
1=\chi_{\alpha_0+\alpha_1}(xy)\cdot \chi_{\alpha_0+2\alpha_1}(xy^2)&\text{if $\alpha_0,\alpha_1$ are shallow}\\
1=\chi_{\alpha_0+\alpha_1+\alpha_2}(xy)&\text{if $\alpha_2,\alpha_0+\alpha_1$ are shallow}\\
1=\chi_{\alpha_0+\alpha_1+\alpha_2}(xy)&\text{if $\alpha_0,\alpha_1+\alpha_2$ are shallow}
\end{cases}
\end{equation}
for all $x,y\in\frakf$. 

\begin{example}
Suppose that the residue field of $k$ has order $q=2$, and let $\lambda$ be the barycenter of the open alcove. Then consider the
shallow character 
\[
\chi:P_+/P_1\to\bbC^\times
\]
given by additive characters 
{\small\[
\begin{array}{|c|c|}
\hline
\alpha&\chi_\alpha(1)\\\hline\hline
\alpha_0&-1\\
\alpha_1&+1\\
\alpha_2&+1\\\hline
\alpha_0+\alpha_1&-1\\
\alpha_1+\alpha_2&-1\\\hline
\alpha_0+2\alpha_1&-1\\
\alpha_0+\alpha_1+\alpha_2&+1\\
2\alpha_1+\alpha_2&-1\\\hline
\end{array}
\]}
Note that $\chi$ has depth $3/4$, but if
\[
n_1=\begin{bmatrix}
&1\\-1\\&&&-1\\&&1
\end{bmatrix}\in N_G(T)
\]
is a lift of the simple reflection about the vanishing hyperplane of $\alpha_1$, then for any $s>3/4$
\[
\sfV_{n_1\lambda,s}\subset\frakf \sfv_{\alpha_0+\alpha_1+\alpha_2},
\]
over which $\chi$ vanishes identically; thus $\chi$ does not statisfy condition $(*)$ in Proposition~\ref{ReederYu type prop}. Despite this, we see that $\chi$ compactly induces to give a supercuspidal representation of $\bfSp_4(k)$. To see this, we first make the following observations:
\begin{itemize}
\item If $\alpha$ is a short affine root, then $n\alpha$ is also short for all $n\in N_G(T)$.
\item The only positive, short affine roots $\alpha$ for which $\chi_\alpha(1)=-1$ are $\alpha_0+\alpha_1$ and $\alpha_1+\alpha_2$.
\item For any $n\in N_G(T)$, either $n(\alpha_0+\alpha_1)$ or $n(\alpha_1+\alpha_2)$ is a positive affine root. 
\end{itemize}
Consequently, for any $n\in N_G(T)$, 
\[
{}^n\chi=\chi\text{ on }{}^nP_+\cap P
\]
only if $n$ either fixes both $\alpha_0+\alpha_1$ and $\alpha_1+\alpha_2$ or swaps them. If $n$ fixes both short affine roots, then either
\[
\left\{\begin{array}{lcl}
n(\alpha_0)&=&\alpha_0-2m\\
n(2\alpha_1+\alpha_2)&=&2\alpha_1+\alpha_2+2m
\end{array}\right\}\quad\text{or}\quad
\left\{\begin{array}{lcl}
n(\alpha_0)&=&2\alpha_0+\alpha_1-2m\\
n(2\alpha_1+\alpha_2)&=&\alpha_2+2m
\end{array}\right\}
\]
holds for some $m\in\bbZ$; if $n$ swaps the short affine roots, then either
\[
\left\{\begin{array}{lcl}
n(\alpha_0)&=&2\alpha_1+\alpha_2-2m+1\\
n(\alpha_0+2\alpha_1)&=&\alpha_2+2m+1
\end{array}\right\}\quad\text{or}\quad
\left\{\begin{array}{lcl}
n(\alpha_0)&=&\alpha_2-2m+1\\
n(\alpha_0+2\alpha_1)&=&2\alpha_1+\alpha_2+2m+1
\end{array}\right\}
\]
holds for some $m\in\bbZ$. In all cases, if $n$ does not act trivially on the affine roots, there exists some long shallow affine root $\alpha$ such that $n\alpha$ is also a positive affine root with
\[
-1=\chi_\alpha(1)\neq\chi_{n\alpha}(1)=1.
\]
Thus, given any $n\in N_G(T)$, there exists some positive affine root $\alpha$ such that $\chi_\alpha(1)\neq\chi_{n\alpha}(1)$. Finally, the affine Bruhat decomposition
\[
G=PN_G(T)P
\]
implies that $\scrI(G,P_+,\chi)=P_\chi=P_+$, where the last equality holds since $q=2$. Hence, we have constructed a supercuspidal representation $\pi(\chi)$ of $\Sp_4(k)$ coming from a shallow character of $I$ not satisfying condition $(*)$ in Proposition~\ref{ReederYu type prop}.
\end{example}

\bibliographystyle{plain}
\bibliography{bibliography}

\end{document}

%% file: preamble.tex
\usepackage{amsmath}
\usepackage{amsthm}
	\theoremstyle{plain}
		\newtheorem{theorem}{Theorem}
		\newtheorem{proposition}[theorem]{Proposition}
		\newtheorem{lemma}[theorem]{Lemma}
		\newtheorem{corollary}[theorem]{Corollary}
		\newtheorem*{theorem*}{Theorem}
		\newtheorem*{proposition*}{Proposition}
		\newtheorem*{lemma*}{Lemma}
	\theoremstyle{definition}
		
		\newtheorem*{definition*}{Definition}
		\newtheorem{example}[theorem]{Example}
		\newtheorem*{example*}{Example}
		
		\newtheorem*{examples*}{Examples}
	\theoremstyle{remark}
		\newtheorem*{remark}{Remark}
		
\usepackage{amssymb}
\usepackage{amsfonts}

\usepackage{enumitem}
\usepackage{array,multirow}
\usepackage[margin=1in]{geometry}
\usepackage{mathtools}
\usepackage{bm}
\usepackage{nccmath}
\usepackage[title]{appendix}
\usepackage{tikz}
	\usetikzlibrary{decorations.markings}
\usepackage{tikz-cd}
\usepackage{float}
\usepackage{subcaption}
\usepackage[scr]{rsfso}

\newcommand{\bbC}{\mathbb{C}}

\newcommand{\bbQ}{\mathbb{Q}}
\newcommand{\bbR}{\mathbb{R}}

\newcommand{\bbZ}{\mathbb{Z}}

\newcommand{\calA}{\mathcal{A}}
\newcommand{\calB}{\mathcal{B}}

\newcommand{\calF}{\mathcal{F}}

\newcommand{\frakf}{\mathfrak{f}}

\newcommand{\frako}{\mathfrak{o}}
\newcommand{\frakp}{\mathfrak{p}}

\newcommand{\frakF}{\mathfrak{F}}

\newcommand{\bfA}{\mathbf{A}}
\newcommand{\bfB}{\mathbf{B}}
\newcommand{\bfC}{\mathbf{C}}

\newcommand{\bfG}{\mathbf{G}}

\newcommand{\bfT}{\mathbf{T}}
\newcommand{\bfU}{\mathbf{U}}

\newcommand{\bfr}{\mathbf{r}}

\newcommand{\sfV}{\mathsf{V}}

\newcommand{\sfv}{\mathsf{v}}

\newcommand{\scrA}{\mathscr{A}}

\newcommand{\scrI}{\mathscr{I}}

\newcommand{\red}{\color{red}}

\DeclareMathOperator{\GL}{GL}

\DeclareMathOperator{\Sp}{Sp}

\DeclareMathOperator{\bfSL}{\bf SL}

\DeclareMathOperator{\bfSp}{\bf Sp}

\DeclareMathOperator{\Hom}{Hom}

\DeclareMathOperator{\End}{End}

\DeclareMathOperator{\ind}{ind}

\DeclareMathOperator{\Span}{span}

\renewcommand{\subset}{\subseteq}